\documentclass[12pt,draft]{amsart}
\usepackage{amssymb}
\usepackage{enumitem}
\usepackage{graphicx}
\usepackage{amssymb}
\usepackage{enumitem}
\usepackage{graphicx}
\usepackage{enumerate, latexsym, amsmath, amsfonts, amssymb, amsthm, graphicx, color, amsmath,
	amscd,color}
\usepackage[all]{xy}
\def\F{\Bbb F}

\def\bg{\bigg}
\def\({\bg(}
\def\){\bg)}

\def\Ack{\medskip\noindent {\bf Acknowledgments}}
\theoremstyle{plain}
\newtheorem{theorem}{Theorem}

\newtheorem{lemma}{Lemma}

\newtheorem{conjecture}{Conjecture}
\theoremstyle{definition}

\theoremstyle{remark}

\makeatletter
\@namedef{subjclassname@2010}{%
	\textup{2010} Mathematics Subject Classification}
\makeatother
\vspace{4mm}

\begin{document}
	\baselineskip=17pt
	\hbox{} {}
	\medskip
	\title[H.-L. Wu and Y.-F. She,]
	{On additive decompositions of primitive elements in finite fields}
	\date{}
	\author[H.-L. Wu and Y.-F. She]{Hai-Liang Wu and Yue-Feng She}
	
	\thanks{2020 {\it Mathematics Subject Classification}.
		Primary 11P70; Secondary 11T24, 11T30.
		\newline\indent {\it Keywords}. primitive elements, additive combinatorics, character sums.
		\newline \indent  This work was supported by the National Natural Science Foundation of China (Grant No. 12101321 and Grant No. 11971222) and the Natural Science Foundation of the Higher Education Institutions of Jiangsu Province (21KJB110002). }
	
	\address {(Hai-Liang Wu) School of Science, Nanjing University of Posts and Telecommunications, Nanjing 210023, People's Republic of China}
	\email{\tt whl.math@smail.nju.edu.cn}
	
	\address {(Yue-Feng She) Department of Mathematics, Nanjing
		University, Nanjing 210093, People's Republic of China}
	\email{{\tt she.math@smail.nju.edu.cn}}
	
	\begin{abstract}
	In this paper, we study several topics on additive decompositions of primitive elemements in finite fields. Also we refine some bounds obtained by  Dartyge and S\'{a}rk\"{o}zy and Shparlinski.
	\end{abstract}
	\maketitle
	\section{Introduction}
	
Let $p$ be an odd prime and let $\mathbb{F}_p$ be the finite field of $p$ elements. Let $\mathbb{F}_p^{\times}$ be the set of all non-zero elements of $\mathbb{F}_p$. It is known that $\mathbb{F}_p^{\times}$ is a cyclic group. An element $x\in\mathbb{F}_p^{\times}$ is called a {\it  primitve element} if $x$ generates $\mathbb{F}_p^{\times}$. Also, we use the symbol $P_p$ to denote the set of all primitive elements in $\mathbb{F}_p$. 

\subsection{Some related results}
Primitive elements have been investiaged extensively in additive combinatorics. For example, Cohen and his collaborators \cite{CET2015} showed that if $p>169$, then there always exists an element
$x\in\F_p^{\times}$ such that $x$, $x+1$ and $x+2$ are all primitive elements.

For any $A_1,\cdots,A_k\subseteq\mathbb{F}_p$, we define 
$$\sum_{i=1}^kA_i=A_1+\cdots+A_k=\{a_1+\cdots+a_k: a_i\in A_i\ \text{for}\ i=1,\cdots,k\}.$$
Dartyge and S\'{a}rk\"{o}zy \cite{Dartyge} studied the additive decompositions of $P_p$. 
In \cite[Conjecture 3]{Dartyge} they posed the following conjecture.
\begin{conjecture}\label{Conj. D and S}
If $p$ is large enough, then 
	$$A+B\neq P_p$$
for any $A,B\subseteq\mathbb{F}_p$ with $|A|,|B|\ge2$. 
\end{conjecture} 
Note that for some small primes $p$ we can find decompostions:
\begin{itemize}
		\item $P_{13}=\{0,5\}+\{2,6\}$.
		\item $P_{19}=\{0,11,12\}+\{2,3,10\}$.
\end{itemize}
This conjecture looks challenging and only some partial results were obtained. Dartyge and S\'{a}rk\"{o}zy \cite[Theorem 2.1]{Dartyge} obtained that if $p$ is large enough and $A+B=P_p$ with $|A|,|B|\ge2$, then
	$$\frac{\varphi(p-1)}{\tau(p-1)\sqrt{p}\log{p}}<|A|,|B|<\tau(p-1)\sqrt{p}\log{p},$$
where $\varphi$ is the Euler function and $\tau(p-1)$ denotes the number of positive divisors of $p-1$.

Moreover, by this result Dartyge and S\'{a}rk\"{o}zy \cite[Theorem 3.1]{Dartyge} showed that if $p$ is large enough, then 
$$A+B+C\neq P_p$$
for any $A,B,C\subseteq\mathbb{F}_p$ with $|A|,|B|,|C|\ge2$.

Shparlinski \cite{Shp} improved Dartyge and S\'{a}rk\"{o}zy's bound and obtain that  
\begin{equation}\label{Eq. Shparlinski}
\frac{\varphi(p-1)}{\sqrt{p}}\ll |A|,|B|\ll\sqrt{p}
\end{equation}
if $A+B=P_p$ with $A,B\subseteq\mathbb{F}_p$ and $|A|,|B|\ge2$. 

\subsection{Main Results}
For any positive integer $n$, let $W(n)$ be the number of positive square-free divisors of $n$. We now state our first result.	
	
\begin{theorem}\label{Thm. A}
	Let $p>3$ be an odd prime. For an integer $k\ge2$, if 
	$$\frac{\varphi(p-1)}{W(p-1)\sqrt{p}}>k2^{k-1},$$
	then $A+B\neq P_p$ for any $A,B\subseteq\mathbb{F}_p$ with $|A|,|B|\ge2$ and $|B|=k$. 
\end{theorem}
	
As our next result, by using the methods introduced by Shkredov \cite{Shkredov} we obtain an improvement of Shparlinski's result.  

\begin{theorem}\label{Thm. B} Let $0<\varepsilon<1/2$ be a real number. If $p$ is a sufficiently large prime and $A+B=P_p$ with $A,B\subseteq\mathbb{F}_p$ with $|A|,|B|\ge2$, then 
	$$\left(\frac{1}{2}-\varepsilon\right)\frac{\varphi(p-1)}{\sqrt{p}}\le |A|,|B|\le 3\sqrt{p}.$$
\end{theorem}

If $|A|=|B|$, we can obtain the following stronger result.

\begin{theorem}\label{Thm. C}
	Let $p$ be an odd prime. If $A+B=P_p$ with $|A|=|B|$, then 
	$$\sqrt{\varphi(p-1)}\le |A|< \sqrt{p}.$$
\end{theorem}

We will prove our main results in sections 2--4 resepectively.

\section{Proof of Theorem \ref{Thm. A}}

Let $p$ be an odd prime. We first introduce some notations. Let $\widehat{\mathbb{F}_p^{\times}}$ be the set of all multiplicative characters of $\mathbb{F}_p$. For each positive divisor $d$ of $p-1$, we define 
$$G_d=\{\chi\in\widehat{\mathbb{F}_p^{\times}}: \chi\ \text{is of order}\ d\}.$$ 
Also, let $\chi_2\in\widehat{\mathbb{F}_p^{\times}}$ be the unique quadratic character, i.e., 
$$\chi_2(x)=
\begin{cases}
	0&\mbox{if}\ x=0,\\
	1&\mbox{if}\ x\ \text{is a non-zero square},\\
	-1&\mbox{otherwise}.
\end{cases}
$$
Clearly $\chi_2(x)=-1$ for any $x\in P_p$. 

Let $P_p(x)$ be the characteristic function of the set $P_p$, i.e., 
$$P_p(x)=
\begin{cases}
	1&\mbox{if}\ x\in P_p,\\
	0&\mbox{otherwise}.
\end{cases}
$$
The explicit formula of $P_p(x)$ is well-known (cf. \cite{CH2003,CH2010}).

\begin{lemma}\label{Lem. characteristic function of P}
Let $p$ be an odd prime. Then 
$$P_p(x)=\frac{\varphi(p-1)}{p-1}\sum_{d\mid p-1}\frac{\mu(d)}{\varphi(d)}\sum_{\chi\in G_d}\chi(x),$$
where $\mu(\cdot)$ is the M\"{o}bius function.
\end{lemma}

\begin{lemma}\label{Lem. H(x)}
Let $p$ be an odd prime and let $k\ge2$ be an integer. Suppose that $A,B\subseteq\mathbb{F}_p$ with $B=\{b_1,b_2,\cdots,b_{k-1},0\}$ and $A+B=P_p$. Then 
\begin{equation}\label{Eq. H(x)}
H(x)=P_p(x)\left(\chi_2(x+b_1)+1\right)\prod_{j=1}^{k-1}\left(\chi_2(x-b_j)+1\right)=0.
\end{equation}
for any $x\in\mathbb{F}_p$. 
\end{lemma}

\begin{proof}
Clearly $H(x)=$ for any $x\not\in P_p$. Suppose now $x\in P_p$. As $A+B=P_p$ and $0\in B$, we have 
$A\subseteq P_p$. If $x\in A$, then $x+b_1\in P_p$. This implies $\chi_2(x+b_1)=-1$ and hence $H(x)=0$. If $x\not\in A$, then at least one of $x-b_1,\cdots,x-b_{k-1}$ is contained in $A\subseteq P_p$. This implies $\chi_2(x-b_j)=-1$ for some $j\in\{1,2,\cdots,k\}$ and hence $H(x)=0$. 

This completes the proof.
\end{proof}

We also need the following result which is known as the Weil Theorem (cf. \cite[Theorem 5.41]{LN}).
\begin{lemma}\label{Lem. Weil's theorem}{\rm (Weil's Theorem)}
	Let $\psi\in\widehat{\F_p^{\times}}$ be a character of order $m>1$, and let $f(x)\in\F_p[x]$ be a
	monic polynomial which is not of the form $g(x)^m$ for any $g(x)\in\F_p[x]$. Let $r$ be the number of distinct roots of $f(x)$ in the algebraic closure of $\F_p$. Then for any $a\in\F_p$ we have
	\begin{equation*}
		\left|\sum_{x\in\F_p}\psi(af(x))\right|\le (r-1)p^{1/2}.
	\end{equation*}
\end{lemma}

Now we prove our first result. For simplicity,  in our proof we write $\sum_{x}$ instead of $\sum_{x\in\mathbb{F}_p}$. Also, we let $\chi_p$ be a generator of $\widehat{\F_p^{\times}}$. 
	
{\bf Proof of Theorem \ref{Thm. A}.} We prove this by contradiction. Suppose that 
$A+B=P_p$ with $A,B\subseteq\mathbb{F}_p$ with $|A|,|B|\ge2$, $|B|=k$, and $$\frac{\varphi(p-1)}{W(p-1)\sqrt{p}}>k2^{k-1}.$$
	
Note that for any $y\in\mathbb{F}_p$ we have 
$$(A+y)+(B-y)=A+B.$$	
By this observation we may assume $0\in B$ and set $B=\{b_1,\cdots,b_{k-1},0\}$.

Let $H(x)$ be as in (\ref{Eq. H(x)}). Then by Lemma \ref{Lem. H(x)} we have 
$$\sum_{x}H(x)=0.$$
On the other hand, one can verify that 
\begin{align*}
\sum_{x}H(x)
&=\sum_{x}P_p(x)+\sum_{x}P_p(x)\chi_2(x+b_1)\\
&+\sum_{x}\sum_{i=1}^{k-1}\sum_{j_1<\cdots<j_i}P_p(x)\chi_2(x-b_{j_1})\cdots\chi_2(x-b_{j_i})\\
&+\sum_{x}\sum_{i=1}^{k-1}\sum_{j_1<\cdots<j_i}P_p(x)\chi_2(x+b_1)\chi_2(x-b_{j_1})\cdots\chi_2(x-b_{j_i}).
\end{align*}

For the first sum on the right hand side, we have 
\begin{equation}\label{Eq. A in the proof of Thm. A}
	\sum_{x}P_p(x)=\varphi(p-1). 	
\end{equation}

For the second sum on the right hand side, we have 
\begin{equation}\label{Eq. B in the proof of Thm. A}
\sum_{x}P_p(x)\chi_2(x+b_1)
=\frac{\varphi(p-1)}{p-1}\sum_{d\mid p-1}\frac{\mu(d)}{\varphi(d)}\sum_{\chi\in G_d}A(\chi),
\end{equation}
where $A(\chi)=\sum_{x}\chi_2(x+b_1)\chi(x)$. 

We now focus on $A(\chi)$. If $\chi\in G_1$, then $A(\chi)=-\chi_2(b_1)$. If $\chi\in G_2$, then $A(\chi)=-1$. In the case $d\mid p-1$ and $d\ge3$, for each $\chi\in G_d$ there is a positive integer $r$ with $(r,d)=1$ such that $\chi=\chi_p^{\frac{p-1}{d}r}$. One can verify that in this case
$$x^{\frac{p-1}{d}r}(x+b_1)^{\frac{p-1}{2}}\neq f(x)^{p-1}$$
for any $f(x)\in\mathbb{F}_p[x]$. Hence by Lemma \ref{Lem. Weil's theorem} we have 
\begin{equation}\label{Eq. bound of A(chi)}
	|A(\chi)|=\left|\sum_{x}\chi_p\left(x^{\frac{p-1}{d}r}(x+b_1)^{\frac{p-1}{2}}\right)\right|\le \sqrt{p}
\end{equation}
if $\chi\in G_d$ with $d\ge3$. 
	
For the third sum on the right hand side, we have 
\begin{equation}\label{Eq. C in the proof of Thm. A}
\begin{split}
&\sum_{x}\sum_{i=1}^{k-1}\sum_{j_1<\cdots<j_i}P_p(x)\chi_2(x-b_{j_1})\cdots\chi_2(x-b_{j_i})\\
=&\sum_{i=1}^{k-1}\sum_{j_1<\cdots<j_i}\frac{\varphi(p-1)}{p-1}\sum_{d\mid p-1}\frac{\mu(d)}{\varphi(d)}\sum_{\chi\in G_d}B_i^{\bf b}(\chi),
\end{split}
\end{equation}	
where ${\bf b}=(j_1,j_2,\cdots,j_i)$ and $B_i^{\bf b}(\chi)=\sum_{x}\chi(x)\chi_2(x-b_{j_1})\cdots\chi_2(x-b_{j_i})$. 

We now focus on $B_i^{\bf b}(\chi)$. If $d\ge3$, for any $\chi\in G_d$ there is a positive integer $r$ with $(r,d)=1$ such that $\chi=\chi_p^{\frac{p-1}{d}r}$. One can verify that 
$$ x^{\frac{p-1}{d}r}(x-b_{j_1})^\frac{p-1}{2}\cdots(x-b_{j_i})^{\frac{p-1}{2}}\neq f(x)^{p-1}$$
for any $f(x)\in\mathbb{F}_p[x]$ in this case. Hence for each $d\ge3$ we have 
$$|B_i^{\bf b}(\chi)|=\left|\sum_{x}\chi_p\left(x^{\frac{p-1}{d}r}(x-b_{j_1})^\frac{p-1}{2}\cdots(x-b_{j_i})^{\frac{p-1}{2}}\right)\right|\le i\sqrt{p}.$$
Note that the above inequality also holds for the cases $d=1,2$. 	
Hence 
\begin{equation}\label{Eq. Bound of B(chi)}
	|B_i^{\bf b}(\chi)|\le i\sqrt{p}.
\end{equation}

For the fourth sum on the right hand side, we have 
\begin{equation}\label{Eq. D in the proof of Thm. A}
\begin{split}
&\sum_{x}\sum_{i=1}^{k-1}\sum_{j_1<\cdots<j_i}P_p(x)\chi_2(x+b_1)\chi_2(x-b_{j_1})\cdots\chi_2(x-b_{j_i})\\
=&\sum_{i=1}^{k-1}\sum_{j_1<\cdots<j_i}\frac{\varphi(p-1)}{p-1}\sum_{d\mid p-1}\frac{\mu(d)}{\varphi(d)}\sum_{\chi\in G_d}C_i^{\bf b}(\chi),
\end{split}
\end{equation}
where $C_i^{\bf b}(\chi)=\sum_{x}\chi(x)\chi_2(x+b_1)\chi_2(x-b_{j_1})\cdots\chi_2(x-b_{j_i})$. 	
	
If $i\ge2$, or $-b_1\neq b_{j_r}$ for any $1\le r\le i$, or $\chi\in G_d$ for some $d\ge2$, then 
\begin{equation}\label{Eq. Bound of C(chi)}
	|C_i^{\bf b}(\chi)|\le (i+1)\sqrt{p}.
\end{equation}	
Also, if $i=1$, and $-b_1=b_{j_r}$ for some $1\le r\le i$, and $\chi\in G_1$, then 
\begin{equation}\label{Eq. Value of C1(chi1)}
	C_1^{\bf b}=p-2.
\end{equation}
In view of (\ref{Eq. A in the proof of Thm. A})--(\ref{Eq. Value of C1(chi1)}), one can verify that $\sum_{x}H(x)$ is greater than 
\begin{align*}
&\varphi(p-1)-\frac{\varphi(p-1)}{p-1}\left(W(p-1)-2\right)\sqrt{p}\\
-&\frac{\varphi(p-1)}{p-1}\sum_{i=1}^{k-1}\binom{k-1}{i}W(p-1)i\sqrt{p}\\
-&\frac{\varphi(p-1)}{p-1}\sum_{i=1}^{k-1}\binom{k-1}{i}W(p-1)(i+1)\sqrt{p}\\
>&\varphi(p-1)-k2^{k-1}W(p-1)\sqrt{p}>0.
\end{align*}
This contradicts $\sum_{x}H(x)=0$. 

In view of the above, we have completed the proof.\qed
	
\section{Proof of Theorem \ref{Thm. B}.}
Let
$$L(\mathbb{F}_p)=\{f:\ f\ \text{is a complex function on $\mathbb{F}_p$}\}.$$
Let $k$ be a positive integer. For any $f\in L(\mathbb{F}_p)$, the function $C_{k+1}(f): \mathbb{F}_p^k\rightarrow\mathbb{C}$ is defined by
$$C_{k+1}(f)(x_1,\cdots,x_k)=\sum_{x}f(x)f(x+x_1)\cdots f(x+x_k).$$
Also, for any $f,g\in L(\mathbb{F}_p)$ we let
$$(f\circ g)(x)=\sum_{y}f(y)g(x+y).$$
Shkredov \cite[Lemma 2.1]{Shkredov} obtained the following result.
\begin{lemma}\label{Lem. Shkredov on C(f)}
	For any $f,g\in L(\mathbb{F}_p)$ and a positive integer $k$,
	\begin{equation*}
		\sum_{x}(f\circ g)^{k+1}(x)
		=\sum_{x_1,\cdots,x_k\in\mathbb{F}_p}C_{k+1}(f)(x_1,\cdots,x_k)\cdot C_{k+1}(g)(x_1,\cdots,x_k).
	\end{equation*}
\end{lemma}	
	
Recall that $\chi_2$ is the quadratic multiplicative character of $\mathbb{F}_p$. We also have the following result. 
\begin{lemma}\label{Lem. bound of sums involving chi2}
	Let $a_1,a_2,a_3,a_4\in\mathbb{F}_p$ be pairwise distinct elements. Then 
	$$\left|\sum_{x}\chi_2(x+a_1)\chi_2(x+a_2)\chi_2(x+a_3)\chi_2(x+a_4)\right|\le1+2\sqrt{p}.$$
\end{lemma}
\begin{proof}
	Observe that 
	\begin{align*}
	&\sum_{x}\chi_2(x+a_1)\chi_2(x+a_2)\chi_2(x+a_3)\chi_2(x+a_4)\\
	=&\sum_{x}\chi_2(x)\chi_2(x+a_2-a_1)\chi_2(x+a_3-a_1)\chi_2(x+a_4-a_1).
	\end{align*}
Hence we may assume that $a_1=0$ and $a_2,a_3,a_4$ are pairwise distinct non-zero elements. One can verify that 
\begin{align*}
&\sum_{x}\chi_2(x)\chi_2(x+a_2)\chi_2(x+a_3)\chi_2(x+a_4)\\
=&\sum_{x\neq 0}\chi_2\left(1+\frac{a_2}{x}\right)\chi_2\left(1+\frac{a_3}{x}\right)\chi_2\left(1+\frac{a_4}{x}\right)\\
=&\sum_{x\neq 0}\chi_2(1+a_2x)\chi_2(1+a_3x)\chi_2(1+a_4x)\\
=&-1+\sum_{x}\chi_2(1+a_2x)\chi_2(1+a_3x)\chi_2(1+a_4x).
\end{align*}
By Lemma \ref{Lem. Weil's theorem} we have 
	$$\left|\sum_{x}\chi_2(x+a_1)\chi_2(x+a_2)\chi_2(x+a_3)\chi_2(x+a_4)\right|\le1+2\sqrt{p}.$$
This completes the proof.
\end{proof}

We also need the following results involving $\varphi(n)$ obtained by Hatalov\'{a} and \v{S}al\'{a}t \cite{H}.

\begin{lemma}\label{Lem. bounds for varphi(n)}
	Let $n\ge3$ be an integer. Then 
	$$\varphi(n)\ge \frac{\log2}{2}\frac{n}{\log n}.$$
\end{lemma}

For any $A\subseteq\mathbb{F}_p$, let $A(\cdot)$ denote the characteristic function of $A$, i.e., 
$$A(x)=
\begin{cases}
	1&\mbox{if}\ x\in A,\\
	0&\mbox{otherwise}.
\end{cases} $$	

{\bf Proof of Theorem \ref{Thm. B}.} Suppose $A+B=P_p$ with $|A|,|B|\ge2$. Then 
\begin{align*}
\sum_{x\in B}\left(A\circ\chi_2\right)^4(x)
&=\sum_{x\in B}\sum_{y_1,y_2,y_3,y_4}\prod_{j=1}^4A(y_j)\chi_2(x+y_j)\\
&=\sum_{x\in B}\sum_{y_1,y_2,y_3,y_4\in A}\prod_{j=1}^{4}(-1)=|A|^4|B|.
\end{align*}
Hence 
$$|A|^4|B|=\sum_{x\in B}\left(A\circ\chi_2\right)^4(x)\le\sum_{x}\left(A\circ\chi_2\right)^4(x).$$
By Lemma \ref{Lem. Shkredov on C(f)} we have 
$$\sum_{x}\left(A\circ\chi_2\right)^4
=\sum_{x_1,x_2,x_3}C_4(A)(x_1,x_2,x_3)\cdot C_4(\chi_2)(x_1,x_2,x_3).$$
By Lemma \ref{Lem. bound of sums involving chi2} we have 
$$
|C_4(\chi_2)(x_1,x_2,x_3)|\le
\begin{cases}
1+2\sqrt{p}&\mbox{if}\ (x_1,x_2,x_3)\not\in E,\\
p&\mbox{otherwise},
\end{cases}
$$
where 
$$E=\{(x,x,0),(x,0,x),(0,x,x): x\in\mathbb{F}_p^{\times}\}\cup\{(0,0,0)\}.$$
	
In view of the above, we obtain 
$$|A|^4|B|\le (1+2\sqrt{p})|A|^4+3p|A|^2,$$	
and hence 
\begin{equation}\label{Eq. A in the proof of Thm. B}
	|A|^2|B|\le (1+2\sqrt{p})|A|^2+3p.
\end{equation}
Since $|A||B|\ge|A+B|=|P_p|=\varphi(p-1)$, by (\ref{Eq. A in the proof of Thm. B}) we obtain 
\begin{equation}\label{Eq. B in the proof of Thm. B}
	(1+2\sqrt{p})|A|^2-\varphi(p-1)|A|+3p\ge0.
\end{equation}
If $p$ is large enough, then by Lemma \ref{Lem. bounds for varphi(n)} the inequality (\ref{Eq. B in the proof of Thm. B}) implies either 
$$|A|\ge\frac{\varphi(p-1)+\sqrt{\Delta}}{2(1+2\sqrt{p})}$$
or 
$$|A|\le\frac{\varphi(p-1)-\sqrt{\Delta}}{2(1+2\sqrt{p})},$$
where 
$$\Delta=\varphi(p-1)^2-12p(1+2\sqrt{p}).$$
However, by Sharplinski's result (\ref{Eq. Shparlinski}) we have 
$$|A|>\frac{\varphi(p-1)-\sqrt{\Delta}}{2(1+2\sqrt{p})}$$
if $p$ is large enough.  By Lemma \ref{Lem. bounds for varphi(n)} again we observe that 
$$
\lim_{p\rightarrow\infty}\left(
\frac{\varphi(p-1)+\sqrt{\Delta}}{2(1+2\sqrt{p})}\cdot\frac{\sqrt{p}}{\varphi(p-1)}\right)=\frac{1}{2}.
$$
Hence for any $0<\varepsilon<1/2$, if $p$ is large enough, then 
\begin{equation}\label{Eq. C in the proof of Thm. B}
|A|\ge\frac{\varphi(p-1)+\sqrt{\Delta}}{2(1+2\sqrt{p})}
\ge\left(\frac{1}{2}-\varepsilon\right)\frac{\varphi(p-1)}{\sqrt{p}}.
\end{equation}

On the other hand, the inequality (\ref{Eq. A in the proof of Thm. B}), together with (\ref{Eq. C in the proof of Thm. B}) and Lemma \ref{Lem. bounds for varphi(n)}, implies that 
\begin{equation}\label{Eq. D in the proof of Thm. B}
	|B|\le 2\sqrt{p}+\frac{3p}{|A|^2}\le 2\sqrt{p}+\frac{3p^2}{(1/2-\varepsilon)^2\varphi(p-1)^2}<3\sqrt{p}
\end{equation}
if $p$ is large enough. Interchanging $A$ and $B$, one can also obtain the lower bound of $|B|$ and the upper bound of $|A|$. In view of the above, when $p$ is large enough, then 
$$\left(\frac{1}{2}-\varepsilon\right)\frac{\varphi(p-1)}{\sqrt{p}}\le|A|,|B|\le3\sqrt{p}.$$

This completes the proof.\qed

\section{Proof of Theorem \ref{Thm. C}.} 
For any $f,g\in L(\mathbb{F}_p)$, the inner product $\langle f,g\rangle$ of $f,g$ is defined by
$$\langle f,g\rangle=\sum_{x}f(x)\overline{g(x)}.$$
Also, we let 
$$||f||_2=(\langle f,f\rangle)^{1/2}=\left(\sum_{x}|f(x)|^2\right)^{1/2}.$$

Shkredov \cite[Lemma 2.4]{Shkredov} obtained the following result.
\begin{lemma}\label{Lem. Shkredov on inner product}
Let $f,g\in L(\mathbb{F}_p)$. Then 
	$$\langle f\circ\chi_2,g\circ\chi_2\rangle=p\langle f,g\rangle-\langle f,1\rangle\cdot\langle \bar{g},1\rangle,$$
where $\bar{g}\in L(\mathbb{F}_p)$ with $\bar{g}(x)=\overline{g(x)}$. 
\end{lemma}
Now we are in a position to prove our third result. This proof uses the method introduced by Shkredov \cite{Shkredov}.

{\bf Proof of Theorem \ref{Thm. C}.} Suppose $A+B=P_p$ with $|A|=|B|=a$. Then we define a function $r: \mathbb{F}_p\rightarrow \mathbb{Z}$ by 
$$A\circ\chi_2(x)=-aB(x)+r(x)$$
for any $x\in\mathbb{F}_p$. Note that for any $x\in B$ we have 
\begin{align*}
	A\circ\chi_2(x)=\sum_{y}A(y)\chi_2(x+y)=\sum_{x\in A}(-1)=-a=-aB(x).
\end{align*}
Hence $r(x)=0$ for any $x\in B$. Now we compute $||r||_2^2$. By definition 
\begin{align*}
	||r||_2^2
	=\sum_{x\not\in B}r(x)^2
	=\sum_{x\not\in B}\left(A\circ\chi_2\right)^2(x)
	=||A\circ\chi_2||_2^2-\sum_{x\in B}\left(A\circ\chi_2\right)^2(x).
\end{align*}
One can verify that 
\begin{align*}
\sum_{x\in B}\left(A\circ\chi_2\right)^2(x)
&=\sum_{x\in B}\sum_{y,z}A(y)A(z)\chi_2(x+y)\chi_2(x+z)\\
&=\sum_{x\in B}\sum_{y,z\in A}1=a^3. 
\end{align*}
Also, by Lemma \ref{Lem. Shkredov on inner product} we have 
$$||A\circ\chi_2||_2^2=p\langle A,A\rangle-\langle A,1\rangle^2=pa-a^2.$$
In view of the above, we obatin 
$$||r||_2^2=pa-a^2-a^3.$$
As $||r||_2^2\ge0$, we obtain 
$$a^2+a-p\le 0$$
and hence 
$$a\le \frac{-1+\sqrt{1+4p}}{2}<\sqrt{p}.$$

On the other hand, as $A+B=P_p$, we have 
$$a^2=|A||B|\ge|A+B|=|P_p|=\varphi(p-1)$$
and hence 
$$a\ge \sqrt{\varphi(p-1)}.$$

In view of the above, we have completed the proof.\qed

\Ack\ The authors would like to thank Prof. Hao Pan for his encouragement.

\end{document}